\documentclass[12pt]{article}
\usepackage{amsmath,amssymb,amsthm} 
\usepackage{mathrsfs,bm,bbm,enumerate}
\usepackage{graphicx,color,tikz}
\usepackage{caption,subcaption}
\usepackage{babel,geometry}
\usepackage[T1]{fontenc} 
\usepackage{hyperref}
\newtheorem{proposition}{\textbf{Proposition}}
\newtheorem{theorem}{\textbf{Theorem}}

\theoremstyle{definition}
\newtheorem{definition}{\textbf{Definition}}

\providecommand{\abs}[1]{\left\lvert#1\right\rvert}

\providecommand{\abs}[1]{\left\lvert#1\right\rvert}

\def\CF{\widehat{\mathscr{P}}}

\def\One{\mathbbm{1}} 
\def\D{{\mathscr{D}}}
\def\S{{\mathscr{S}}}

\def\R{{\mathscr{R}}}

\def\drm{\mathrm{d}}

\def\C{ \mathbb{C}}
\def\Z{ \mathbb{Z}}
\def\N{ \mathbb{N}}
\def\R{ \mathbb{R}}

\def\Supp{\mathrm{Supp}} 

\begin{document}
\title{On Tempered Discrete and Lévy White Noises}

\author{Julien Fageot\footnote{Julien Fageot was supported by the Swiss National Science Foundation (SNSF) under Grant \texttt{P400P2\_194364}.}}

\date{\'Ecole polytechnique f\'ed\'erale de Lausanne}
 
\maketitle

\begin{abstract}
We study the growth properties of the family of i.i.d.  random sequences, also known as discrete white noises, and of their continuous-domain generalization, the family of Lévy white noises.
%discrete white noises---\emph{i.e.}, independent and identically distributed random sequences---
More precisely, we characterize the members of both families which are tempered---\emph{i.e.}, whose asymptotic growth is dominated by some polynomial---in terms of their moment properties.
We recover the characterization of tempered L\'evy white noises obtained by Robert Dalang and Thomas Humeau and provide a new proof of there fundamental result. 
Our approach is based on a fruitful connection between the discrete and continuous-domain white noises. 
\end{abstract}

%\tableofcontents
 
%%%%%%%%%%%%%%%%%%%%%%%%%%%%%%%%%%%%%%%%%%
\section{Introduction} \label{sec:intro}
%%%%%%%%%%%%%%%%%%%%%%%%%%%%%%%%%%%%%%%%%%
 
    \subsection{Lévy White Noises as Generalized Random Processes}
    
 The Gaussian white noise is a random process whose sample paths are too rough to be described as a pointwise random function.
 For this reason, it has been considered as a random element in the space $\S'(\R)$ of tempered generalized functions, for instance in quantum physics~\cite{Simon1979functional}, stochastic partial differential equations~\cite{walsh1986introduction}, and white noise analysis~\cite{Hida1975,SiSi2012}. 
 
    This construction of the Gaussian white noise relies on the theory of \textit{generalized random processes},initiated independently in the 50's by Israel Moiseevich Gel'fand~\cite{Gelfand1955generalized} and Kiyoshi Itô~\cite{Ito1954distributions}. 
    This framework is the stochastic counterpart of the theory of generalized functions of Laurent Schwartz~\cite{Schwartz1966distributions} and has been extensively detailed in the monographs~\cite{GelVil4,Fernique1967processus,Ito1984foundations}; see also~\cite{Bierme2017generalized} for a recent introduction.
    The key idea is to endow generalized function spaces with random structures compatible with their topology. This allows to define random elements in the space $\S'(\R)$, but also over the larger space $\D'(\R)$ of (not necessarily tempered) generalized functions.  
    The later is used by Gel'fand and Vilenkin to introduce the complete family of Lévy white noises~\cite[Chapter 3]{GelVil4}. Those random processes share the properties of the Gaussian white noise with the exception of Gaussianity in the same way Lévy processes generalize the Brownian motion~\cite{Bertoin1998levy,Sato1994levy}.
 
    Lévy white noises are therefore constructed as random elements in the space $\D'(\R)$, while the Gaussian white noise is known to be tempered, \emph{i.e.} located in the space $\S'(\R)$. It is then natural to ask whether a given Lévy white noise is tempered or not.
 
    \subsection{The Temperedness of Lévy White Noise: Related Works}
    
    Tempered Lévy white noises have several advantages. For instance, as random elements of $\S'(\R)$, there Fourier transform is well-defined and they can therefore be used for stochastic partial differential equations SPDEs. Many works dealing with Lévy white noise therefore identify sufficient conditions so that the temperedness is achieved.
    The authors of~\cite{lee1999segal} consider Lévy white noise with finite first moment, while finite variance Lévy white noises are used for SPDEs in~\cite{Oksendal2004white,oksendal2008stochastic}. 
    
    In~\cite{Fageot2014}, Arash Amini, Michael Unser, and the author of this article have shown that a Lévy white noise is tempered as soon as it possesses a finite absolute moment of order $\epsilon>0$ arbitrarily small, what they called the \emph{$\epsilon$-condition}. 
    They conjectured that this sufficient condition is also necessary. This was further proved by Robert Dalang and Thomas Humeau in \cite{Dalang2015Levy}. More precisely, these authors proved the following result, exposed in Theorem~\ref{theo:DH}, which also reveals that Lévy white noises are almost surely not tempered when the {$\epsilon$-condition} does not occur. There result therefore fully characterize the temperedness of Lévy white noise from simple moments properties.
    
    \begin{theorem}[Dalang and Humeau, Theorem 2.5 in \cite{Dalang2015Levy}]
    \label{theo:DH}
    Let $W$ be a Lévy white noise on $\D'(\R)$. If $\mathbb{E} [|\langle W , \One_{[0,1]} \rangle|^\epsilon] < \infty$\footnote{The random variable $X_\varphi = \langle W , \varphi \rangle$ is a priori well-defined for $\varphi \in \D(\R)$, which is not the case for $\varphi = \One_{[0,1]}$. However, one can extend the definition of $X_\varphi$ to test functions that are compactly supported and bounded for any Lévy white noise, as can easily be deduced from~\cite{fageot2021domain}. This is implicit in Theorem~\ref{theo:DH}.} for some $\epsilon > 0$, then $W \in \S'(\R)$ almost surely. Otherwise, $W \notin \S'(\R)$ almost surely.  
    \end{theorem}
    
    Understanding the sample path properties of random processes is a natural question. The growth properties of Lévy processes has been considered in~\cite{Pruitt1981growth}; see~\cite[Section 48]{Sato1994levy} for extensive discussions and additional references on this matter. The control of the growth of Lévy processes has also been considered by authors studying the regularity of Lévy processes and Lévy white noises, measured in terms of weighted Besov spaces\cite{Schilling1998growth,aziznejad2020waveletbesov,Fageot2017multidimensional}. Those weighted Besov spaces being subspaces of $\S'(\R)$, the temperedness of the considered random processes is established. However, these works do not prove the necessity of the $\epsilon$-condition.  
    
  %  It is worth noting that the sufficiency of the \emph{$\epsilon$-condition} can be deduced from known results regarding the growth properties of Lévy processes, characterized with the Pruitt index introduced by William E. Pruitt~\cite{Pruitt1981growth}. 
 %   Such index has been used for the study of the regularity of Lévy processes~\cite{Schilling1998growth} and Lévy white noises~\cite{aziznejad2020waveletbesov,Fageot2017multidimensional}, in particular by characterizing the \textit{weighted} Besov spaces to which they belong. Those weighted Besov spaces being subspaces of $\S'(\R)$, the temperedness of the considered random processes is established. However, these works do not prove the necessity of the $\epsilon$-condition.  
 %   Understanding the sample path properties of random processes, such as its asymptotic properties, is a natural question. For instance, the growth properties of Lévy processes has been considered in~\cite{Pruitt1981growth}; see~\cite[Section 48]{Sato1994levy} for extensive discussions and additional references on this matter.
 %   The control of the growth of Lévy processes has also been considered by authors studying the regularity of Lévy processes and Lévy white noises. The typical question is to know in which function space is a given random process. 

     \subsection{Contributions and Outline}
     
     Our goal in this paper is to revisit the characterization of tempered Lévy white noise of Dalang and Humeau. In particular, we provide a new proof of Theorem~\ref{theo:DH}. The exposition is self-contained and highlights the connection between Lévy white noises and families of i.i.d. random sequences (\emph{i.e.}, discrete white noises). The paper is organized as follows.
     
     \begin{itemize}
         \item We first focus on discrete white noises in Section~\ref{sec:discrete}. Our main result is exposed in Theorem~\ref{theo:discrete}: We characterize the weighted $\ell_\infty$-sequence spaces on which an i.i.d. random sequence belongs in terms of its moment properties. In particular, we demonstrate that the $\epsilon$-condition is equivalent to the temperedness of \textit{discrete} white noises. Even if this result is not surprising, it has never been observed to the best of our knowledge.
         
         \item We then consider the case of (continuous-domain) Lévy white noises in Section~\ref{sec:levynoise}. Our main result is Theorem~\ref{theo:lévynoisetempered}, where the temperedness of Lévy white noise is characterized from different perspectives. 
         It covers Theorem~\ref{theo:DH}, together with known characterizations of the temperedness in terms characteristic functional, characteristic exponent, and the Lévy measure of the underlying white noise. 
         The novelty of our result is the connection between a Lévy white noise and the random sequences $(\langle W , \varphi(\cdot- n)\rangle)_{n\in \Z}$, and results in a new proof of Theorem~\ref{theo:DH}. 
     \end{itemize}

\section{Discrete White Noises} \label{sec:discrete}
%%%%%%%%%%%%%%%%%%%%%%%%%%%%%%%%%%%%%%%%%%

    %------------------------------------------
    \subsection{Sequence Spaces} \label{sec:sequencespaces}
    %------------------------------------------

    We denote by $\D'(\Z)$ the space of all real sequences. 
    This notation is somehow unusual and is made in analogy with the space $\D'(\R)$ of generalized functions (see Section~\ref{sec:functionspaces}).
    The space $\D'(\Z)$ is the  dual of the space $\D(\Z)$ of sequences that have finitely many nonzero elements. Any sequence $u\in \D'(\Z)$ indeed defines a linear form over the space $\D(\Z)$ via the relation
    \begin{equation} \label{eq:dualprod}
        \langle u ,v \rangle = \sum_{n\in \Z} u[n] v[n]
    \end{equation} 
    for any $v \in \D(\Z)$. Note that the sum in \eqref{eq:dualprod} is actually finite. 
    One can show that $\D'(\Z)$ is the \textit{topological} dual of $\D(\Z)$ if we endow the later with the locally convex topology corresponding to the following notion of convergence. The sequences $v_k = (v_k[n])_{n\in \Z} \in \D(\Z)$, $k \geq 1$, converge to $0$ if they have a common support $K \subset \Z$ outside of which they vanish and if the finite-dimensional vectors $(v_k[n])_{n\in K}$ converge to $0$. \\
    
   The  space $\ell_\infty(\Z)$ of bounded sequences $u = (u[n])_{n\in \Z} \in \D'(\Z)$ such that $\lVert u \rVert_\infty = \sup_{n\in\Z} |u[n]| < \infty$ is a Banach space for the norm $\lVert \cdot \rVert_\infty$. 
    More generally, for $\alpha \in \R$, let $\ell_{\infty,\alpha}(\Z)$ be the weighted Banach space for the norm 
    \begin{equation}
        u \mapsto \lVert u \rVert_{\infty,\alpha} = \sup_{n\in\Z} |n|^\alpha |u[n]| < \infty.
    \end{equation}  
    When $\alpha _1 \geq \alpha_2$, we have the topological embedding $\ell_{\infty,\alpha_1}(\Z) \subseteq \ell_{\infty,\alpha_2}(\Z)$. \\

    The space of rapidly decaying sequence is denoted by $\S(\Z)$. It corresponds to the sequences $u = (u[n])_{n\in \Z}$ such that $|n|^\alpha |u[n]| \rightarrow_{|n|\rightarrow \infty} 0$ for any $\alpha \geq 0$. It is a Fr\'echet space for the family of norms $\lVert \cdot \rVert_{\infty,\alpha}$ with $\alpha \geq 0$ and we have the projective limit
    \begin{equation}
        \S(\Z) = \bigcap_{\alpha \geq 0} \ell_{\infty,\alpha}(\Z). 
    \end{equation}
    The topological dual of $\S(\Z)$ is the space $\S'(\Z)$ of tempered sequences, \emph{i.e.}, the space of sequences which are bounded by a polynomial. It is the inductive limit 
    \begin{equation}
        \S'(\Z) = \bigcup_{\alpha \geq 0} \ell_{\infty,- \alpha}(\Z). 
    \end{equation}
    The following simple characterization of tempered sequences will be useful in Section~\ref{sec:lévy}. 
    
    \begin{proposition} \label{prop:caracSpinDp}
    Let $u \in \D'(\Z)$ seen as a linear functional over $\D(\Z)$ according to \eqref{eq:dualprod}. Then, $u \in \S'(\Z)$ if and only if $u : \D(\Z) \rightarrow \R$ is continuous for the topology induced by $\S(\Z)$. This is equivalent to the fact that 
    \begin{equation}
        \langle u , v_k \rangle \underset{k \rightarrow \infty}{\longrightarrow} 0
    \end{equation}
    for any sequences $v_k \in \D(\Z)$, $k\geq 1$, such that $\lVert v_k \rVert_{\infty,\alpha} \rightarrow 0$ when $k\rightarrow \infty$ for any $\alpha \geq 0$. 
    \end{proposition}

    \begin{proof}
        By definition, $u \in \S'(\Z)$ defines a continuous and linear form over $\S(\Z)$ as in \eqref{eq:dualprod} and its restriction over $\D(\Z)$ for the topology induced by $\S(\Z)$ is clearly continuous. Conversely, any $u \in \D'(\Z)$ such that $v \in \D(\Z) \mapsto \langle u , v \rangle$ is continuous for the topology of $\S(\Z)$ and can be uniquely extended as a linear and continuous form over $\S(\Z)$ (using that $\D(\Z)$ is dense in $\S(\Z)$) and the extension is therefore in $\S'(\Z)$ by definition of the dual. This extension corresponds to the sequence $u$ itself, which is therefore tempered.
    \end{proof}

    %------------------------------------------
    \subsection{Random Sequences and Discrete White Noises} \label{sec:sequencespaces}
    %------------------------------------------
    Throughout this paper, we fix a probability space $(\Omega, \mathcal{F}, \mathscr{P})$.
    A random sequence is a random element in the space $\D'(\Z)$. More precisely,  it is a mapping $X : \Z \times \Omega \rightarrow \R$ such that $X(n,\omega) = X_n(\omega) \in \R$ for any $\omega \in \Omega$ and $n\in \Z$. Then, for any fixed $n\in \Z$, $X_n$ is  a real random variable. For any $\omega \in \Omega$, we have that $X(\omega) = (X_n(\omega))_{n\in \Z} \in \D'(\Z)$. 
    The law of $X$ is characterized by the quantities
    \begin{equation}
        \mathscr{P}_X( \{ u \in \D'(\Z) | \langle u ,v \rangle \leq a \} ) = \mathscr{P} ( \langle X, v \rangle \leq a)
    \end{equation}
    for any $v \in \D(\Z)$ and $a \in \R$. Thereafter, it will be important to characterize the random sequences $X$ that are \textit{tempered}, meaning that $X \in \S'(\Z)$ almost surely. 
    
    \begin{definition}
    A \emph{discrete white noise} is a random sequence $X = (X_n)_{n\in \Z}$ of independent and identically distributed (i.i.d.) random variables $X_n$.
    \end{definition}
    
    The law of a discrete white noise $X$ is fully characterized by the law of $X_0$. For any sequence $(a_n)_{n\in \Z}$, we indeed have that
    \begin{equation}
        \mathscr{P}( \forall n \in \Z, \ X_n \leq a_n ) = \prod_{n\in \Z} \mathscr{P}( X_n \leq a_n) = \prod_{n\in \Z} \mathscr{P}( X_0 \leq a_n),
    \end{equation}
    where the infinite products are converging in $[0,1]$ (as any infinite product of terms in $[0,1]$). 
    A discrete white noise is \emph{a priori} in $\D'(\Z)$. In the next section, we provide characterizations of discrete white noises to be localized in the sequence spaces $\ell_{\infty,\alpha}(\Z)$ and $\S'(\Z)$. 

    %------------------------------------------
    \subsection{Characterization of Tempered Discrete White Noises} \label{sec:sequencespaces}
    %------------------------------------------
    
    The goal of this section is to study the growth properties of discrete white noises. More precisely, we shall prove the following result.
    
    \begin{theorem}
    \label{theo:discrete}
    Let $X = (X_n)_{n\in \Z}$ be a discrete white noise. Then, we have the equivalence
    \begin{equation}
        X \in \ell_{\infty,-1/p} (\Z) \text{ a.s. } \quad \Longleftrightarrow  \quad  \mathbb{E} [|X_0|^p] < \infty. 
    \end{equation}
    It this condition is not fulfilled, then  $X \notin \ell_{\infty,-1/p} (\Z)$ a.s. We also have the equivalence
    \begin{equation}
        X \in \S'(\Z) \text{ a.s. } \quad \Longleftrightarrow  \quad  \exists \ \epsilon > 0, \  \mathbb{E} [|X_0|^\epsilon] < \infty. 
    \end{equation}
    If this condition is not fulfilled, then  $X \notin \S'(\Z)$ a.s.
    \end{theorem}
    
    Our proof is based on Proposition~\ref{prop:zeroone}, which establishes $(0-1)$ laws for discrete white noises.
    
    \begin{proposition} \label{prop:zeroone} 
    For any discrete white noise $X$ and any $\alpha > 0$, we have that
    \begin{equation}
        \mathscr{P} ( X \in \ell_{\infty,-\alpha}(\Z) ) \in \{0, 1 \}
    \end{equation}
    and $\mathscr{P} ( X \in \ell_{\infty,-\alpha}(\Z) ) = 1$ if and only if 
    \begin{equation} \label{eq:conditionX0}
        \sum_{n \geq 0} \mathscr{P}( |X_0| \geq n^{\alpha} ) < \infty.
    \end{equation}
    Moreover, $ \mathscr{P} ( X \in \S'(\Z) ) \in \{0, 1 \}$, the value being $1$ if and only if there exists $\alpha \geq 0$ such that \eqref{eq:conditionX0} holds. 
    \end{proposition}

    \textit{Remark.} The characterization \eqref{eq:conditionX0} in Proposition~\ref{prop:zeroone} is false for $\alpha = 0$. 
    In this case, we have that $\mathscr{P}(X \in \ell_{\infty}(\Z) ) = 1$ if and only if the common law $\mathscr{P}_{X_0}$ of the $X_n$ has a finite support. 
 
    \begin{proof}[Proof of Proposition~\ref{prop:zeroone}]
    For $x \geq 0$, we set $F(x) = \mathscr{P}(|X_0| \geq x)$. The set $\ell_{\infty,-\alpha}(\Z)$ satisfies
    \begin{equation}
        \ell_{\infty,-\alpha}(\Z) = \bigcup_{C > 0} \left\{ u = (u_n)_{n\in \Z} \in \D'(\Z), \ |u_n||n|^\alpha \leq C \right\},
    \end{equation}
    where the union over $C > 0$ is increasing when $C$ increases. Hence, we deduce that 
    \begin{align} \label{eq:Ineedthat}
        \mathscr{P} ( X \in  \ell_{\infty,-\alpha}(\Z) )
        &=
        \lim_{C\rightarrow \infty} \mathscr{P}( \forall n \in \Z, \ |X_n||n|^{-\alpha} \leq C )
        =
        \lim_{C\rightarrow \infty} \prod_{n \in \Z} \mathscr{P}( |X_0||n|^{-\alpha} \leq C ) \nonumber \\
       & = 
        \lim_{C\rightarrow \infty} \prod_{n \in \Z} \left(1 - F(C|n|^\alpha) \right)
        = 
        \lim_{C\rightarrow \infty} \exp\left( \sum_{n \in \Z} \log\left(1 - F(C|n|^\alpha) \right) \right),
    \end{align}
    where the infinite product converges in $[0,1]$ and the infinite sum in $[-\infty, 0]$. 
    
    If $(u_n)_{n\in \Z}$ is a sequence such that $0 \leq u_n < 1$ for any $n\in \Z$, and $u_n \rightarrow 0$ when $|n|\rightarrow \infty$, then $- \log (1 - u_n) \geq 0$ for every $n \in \Z$ and
    \begin{equation} \label{eq:usefulcomparaison}
        \sum_{n\in \Z}  u_n < \infty \quad  \Longleftrightarrow   \quad  \sum_{n\in \Z}  - \log (1 - u_n) < \infty.
    \end{equation}
    This follows from the fact that $- \log ( 1 - u_n) = \log ( 1 + \frac{u_n}{1-u_n}) \sim \frac{u_n}{1-u_n} \sim u_n$ when $|n|\rightarrow \infty$.
    We apply this relation to $u_n = F(C|n|^\alpha)$ and  distinguish two cases. \\
    
    {Assume first that  $\sum_{n\in \Z} F(C|n|^\alpha) = \infty$ for any $C > 0$.} Then, \eqref{eq:usefulcomparaison} implies that 
    \begin{equation}
        \sum_{n \in \Z} \log\left(1 - F(C|n|^\alpha) \right) = - \infty
    \end{equation}
    for any $C >0$ and therefore, due to \eqref{eq:Ineedthat}, $\mathscr{P} ( X \in  \ell_{\infty,-\alpha}(\Z) ) = \exp( - \infty ) = 0$. \\
    
    {Assume now that there exists $C_0 > 0$ such that $\sum_{n\in \Z} F(C_0|n|^\alpha) = \infty$.}  
    %Again according to \eqref{eq:usefulcomparaison}, we have that $\sum_{n \in \Z} - \log\left(1 - F(C_0|n|^\alpha) \right) < \infty$. 
    Then, we have that
    \begin{equation}
        - \log (1 - F(C|n|^\alpha)) \underset{C\rightarrow \infty}{\longrightarrow}  - \log (1 - F(\infty)) = - \log (1-0) = 0. 
    \end{equation}
    Moreover, for $C \geq C_0$, using that $F$ is decreasing, 
    \begin{equation}
    - \log (1 - F(C|n|^\alpha)) \leq - \log (1 - F(C_0|n|^\alpha)),
    \end{equation}
    the right term being summable (using \eqref{eq:usefulcomparaison}) 
    and independent from $C \geq C_0$. We can therefore apply the Lebesgue dominated convergence theorem to deduce that
    \begin{equation}
        \sum_{n\in \Z}    - \log (1 - F(C|n|^\alpha)) \underset{C\rightarrow \infty}{\longrightarrow} 0.
    \end{equation}
    Finally, this implies with \eqref{eq:Ineedthat} that  $\mathscr{P} ( X \in  \ell_{\infty,-\alpha}(\Z) ) = \exp(0) = 1$. \\
    
    We have shown that  $\mathscr{P} ( X \in  \ell_{\infty,-\alpha}(\Z) )$ can only take the value $0$ and $1$. Moreover, the value $1$ is taken if and only if there exists $C_0 > 0$ such that $\sum_{n\in \Z} F(C_0|n|^\alpha) = \infty$ holds. However, this condition is then true for any $C_0$. To see this, we simply observe that 
    \begin{equation} \label{eq:similarsumintegral}
        \sum_{n\in \Z} F(C_0|n|^\alpha) < \infty 
        \quad \Longleftrightarrow \quad
        \int_{\R} F( C_0 |x|^\alpha ) \mathrm{d} x < \infty. 
    \end{equation}
    Indeed, $F$ is decreasing and therefore, for any $0 \leq n \leq x \leq n+1$, $F(C_0 (n+1)^\alpha) \leq F(C_0 x^\alpha) \leq F(C_0 n^\alpha)$ (with similar relations for $x \leq 0$), which easily implies \eqref{eq:similarsumintegral}. Moreover, a change of variable $x \leftarrow C_0^{1/\alpha} x$ gives that
    $\int_{\R} F( C_0 |x|^\alpha ) \mathrm{d} x= C_0^{-1 / \alpha} \int_{\R} F( |x|^\alpha ) \mathrm{d} x$, hence the finiteness does not depend on $C_0 > 0$. Finally, \eqref{eq:conditionX0} follows by choosing $C_0=1$. \\
    
    For $\S'(\Z)$, we remark that the spaces $\ell_{\infty,-\alpha}(\Z)$ are growning when $\alpha \geq 0$ increases, hence
    \begin{equation}
        \S'(\Z) = \bigcup_{\alpha \geq 0} \ell_{\infty,-\alpha} (\Z) = \lim_{\alpha\rightarrow \infty} \ell_{\infty,-\alpha}(\Z).
    \end{equation}
%    where the union is increasing when $\alpha$ increases. 
    This implies that
    \begin{equation}\label{eq:usefulagain}
        \mathscr{P} ( X \in \S'(\Z) )= \lim_{\alpha \rightarrow \infty} \mathscr{P} ( X \in \ell_{\infty,-\alpha} (\Z) ) . 
    \end{equation}
    Assume that there exists $\alpha_0 > 0$ such that \eqref{eq:conditionX0} holds. Then, $ \mathscr{P} ( X \in  \ell_{\infty,-\alpha_0}(\Z) ) = 1$ 
    and since $\ell_{\infty,-\alpha_0} (\Z) \subset \ell_{\infty,-\alpha} (\Z)$, 
    we also have that $\mathscr{P} ( X \in  \ell_{\infty,-\alpha}(\Z) ) = 1$ for any $\alpha \geq \alpha_0$, 
    hence $\mathscr{P} ( X \in \S'(\Z) = 1$ due to \eqref{eq:usefulagain}. 
    On the contrary, assume that $\sum_{n \geq 0} \mathscr{P}( |X_0| \geq n^{\alpha} ) = \infty$ for any $\alpha$. 
    This implies that $\mathscr{P} ( X \in \ell_{\infty,-\alpha} (\Z) ) = 0$ for every $\alpha \geq 0$ and therefore $\mathscr{P} ( X \in \S'(\Z) ) = 0$, again due to \eqref{eq:usefulagain}.
    \end{proof}
    
    \begin{proof}[Proof of Theorem~\ref{theo:discrete}] 
    We first observe that the $(0-1)$ law in Proposition \ref{prop:zeroone} implies that $X \notin \ell_{\infty, -1/p}(\Z)$ a.s. as soon as the condition $X \in \ell_{\infty, -1/p}(\Z)$ is not satisfied (the same holds for the sequence space $\S'(\Z)$). This covers the last part of Theorem~\ref{theo:discrete}. 
    
    Using the same argument as in the proof of Proposition~\ref{prop:zeroone}, we moreover have that 
    \begin{equation} \label{eq:1stuff}
           \sum_{n \geq 0} \mathscr{P}( |X_0| \geq n^{1/p} ) = \sum_{n \in \Z} F( n^{1/p}) < \infty 
           \quad \Longleftrightarrow \quad
           \int_{0}^{\infty} F( x^{1/p} ) \drm x < \infty.
    \end{equation}
    Denoting by $\mathscr{P}_{|X_0|}$ the probability law of $|X_0|$, we also remark that
    \begin{align} \label{eq:2stuff}
         \int_{0}^{\infty} F( x^{1/p} ) \drm x  &= 
         \int_0^\infty \int_0^\infty \One_{t \geq x^{1/p}}
 \mathscr{P}_{|X_0|}(\drm t) \drm x  \
        =
        \int_0^\infty\left(  \int_0^\infty \One_{t^p \geq x} \drm x \right) \mathscr{P}_{|X_0|}(\drm t) \nonumber \\
        &=
        \int_0^\infty t^p \mathscr{P}_{|X_0|}(\drm t) = \mathbb{E} [|X_0|^p].
    \end{align}
    Finally, this shows that $\sum_{n \geq 0} \mathscr{P}( |X_0| \geq n^{1/p} )  < \infty$ if and only if $\mathbb{E} [|X_0|^p] < \infty$ and Theorem~\ref{theo:discrete} follows from the criterion of Proposition~\ref{prop:zeroone}.
    \end{proof}
        
    We illustrate Theorem~\ref{theo:discrete} with some examples. Thereafter, $X= (X_n)_{n\in \Z}$ is a discrete white noise. 
    \begin{itemize}
        \item Assume that the random variables $X_n$, $n\in \Z$, are symmetric-$\alpha$-stable (S$\alpha$S)~\cite{Taqqu1994stable}. This means that the common characteristic function of the $X_n$ is $\CF_{X_n}(\xi) = \exp ( - \gamma^\alpha |\xi|^\alpha)$ where $\gamma > 0$ and $0 < \alpha \leq 2$. The case $\alpha = 2$ corresponds to the Gaussian law, for which all moments are finite. Hence, $X \in \ell_{\infty,-\epsilon}(\Z)$ for any $\epsilon > 0$ in this case. 
        
        Non-Gaussian S$\alpha$S random variables have an infinite variance: For $\alpha < 2$, we actually have that $\mathbb{E} [|X_0|^p] < \infty$ if and only if $p<\alpha$~\cite[Property 1.2.16]{Taqqu1994stable}. Hence, Theorem~\ref{theo:discrete} implies that, for $0< \alpha < 2$,
        \begin{equation}
            X \in \ell_{\infty, - \frac{1}{\alpha} - \epsilon}(\Z)
        \end{equation}
        almost surely for any $\epsilon > 0$. In particular, $X \in \S'(\Z)$ almost surely. Moreover, the fact that  $\mathbb{E} [|X_0|^\alpha] = \infty$ implies that $ X \notin \ell_{\infty, - \frac{1}{\alpha}}(\Z)$ almost surely. 
        
        \item As for the Gaussian case, random variables whose moments are all finite are such that $X \in \ell_{\infty, -\epsilon}(\Z)$ almost surely for any $\epsilon>0$. As such, they only slightly differ from a bounded sequence, since $\sup_{n\in\Z} \frac{|X_n|}{1+|n|^\delta} < \infty$ almost surely for $\delta$ arbitrarily small. 
        
        \item Assume that the $X_n$ have a common probability density function given by $f(x) = \frac{c}{(|x|+1)\log^2 (|x|+2)}$ where $c > 0$ is such that $\int_{\R} f(x) \drm x = 1$. Note that $f$ is indeed integrable as a Bertrand integral. However, we have for any $p > 0$ that
        \begin{equation}
            \mathbb{E} [|X_0|^p] = c \int_{\R} \frac{|x|^p}{(|x|+1)\log^2 (|x|+2)} \drm x = \infty
        \end{equation}
        due to the fact the integrand behaves asymptotically like $|x|^{p-1} / \log^2 (|x|)$ and is therefore non-integrable for $p>0$. Hence, $X \notin\S'(\Z)$ almost surely. 
    \end{itemize}
    
%%%%%%%%%%%%%%%%%%%%%%%%%%%%%%%%%%%%%%%%%%
\section{Lévy White Noises} \label{sec:lévy}
%%%%%%%%%%%%%%%%%%%%%%%%%%%%%%%%%%%%%%%%%%
    
    %------------------------------------------
    \subsection{Spaces of Generalized Functions} \label{sec:functionspaces}
    %------------------------------------------    
     
     We briefly introduce the relevant function spaces that will be used thereafter. More details can be found, for instance, in~\cite{friedlander1998introduction}. We refer the interested reader to~\cite{Treves1967} for a comprehensive exposition on the topological structures of the proposed function spaces. 
     
    Let $\D(\R)$ be the space of compactly supported and infinitely smooth functions. It is equipped with its natural inductive topology as the union of the spaces $\D([-n,n])$ of infinitely smooth functions whose support is included in $[-n,n]$ for $n \geq 1$. In particular, a family of functions $(\varphi_k)_{k\in \Z}$ in $\D(\R)$ converges to $0$ for this topology if there exists $n \geq 1$ such that $\Supp(\varphi_k) \subset [-n,n]$ for every $k \geq 1$ and $\varphi_k$ converges uniformly to $0$ on $[-n,n]$ together with all its derivatives.
    
    The topological dual $\D'(\R)$ of $\D(\R)$ is the space of generalized functions. A locally integrable function $f$ defines a linear and continuous functional over $\D(\R)$ via the relation
    \begin{equation}
        f : \varphi \mapsto \langle f, \varphi \rangle = \int_{\R} \varphi(x) f(x) \drm x.
    \end{equation}
    The framework of generalized functions allows to define functions with no pointwise interpretation, such as the Dirac impulse $\delta$ and all its derivatives. We will  define Lévy white noise are random elements in $\D'(\R)$ in Section~\ref{sec:levynoise}. \\
    
    The Schwartz space of rapidly decaying and infinitely smooth functions is denoted by $\S(\R)$. A function $f$ is in $\S(\R)$ if it vanishes faster than any polynomial together with all its derivatives. 
    Then, $\S(\R)$ is a Fréchet space for the family of semi-norms
    \begin{equation} \label{eq:seminormsSp}
            \lVert \varphi \rVert_{\infty,\alpha,\beta}  = \sup_{x\in \R}  |x|^\alpha |\mathrm{D}^\beta \varphi  (x)| 
    \end{equation}
    where $\alpha \geq 0$ and $\beta \in \N$. 
    In particular, a sequence $(\varphi_k)_{k \geq 1}$ of functions in $\S(\R)$ converges to $0$ if $\lVert \varphi_k \rVert_{\infty,\alpha,\beta} \rightarrow 0$ as $k\rightarrow \infty$ for any $\alpha \geq 0$ and $\beta \in \N$.
    
    The topological dual of $\S(\R)$ is the space $\S'(\R)$ of tempered generalized functions. It is a subspace of $\D'(\R)$ and is informally made of generalized functions that are slowly growing, \emph{i.e.} growing slower than some polynomial. 
    
    %------------------------------------------
    \subsection{L\'evy White Noises as Generalized Random Processes} \label{sec:levynoise}
    %------------------------------------------

    A generalized random process is a random element of the space $\D'(\R)$ of generalized functions. 
    The theory of generalized random processes was initiated independently by  Gel'fand~\cite{Gelfand1955generalized} and  Itô~\cite{Ito1954distributions}, it offers a broad framework for defining random processes, including some with no pointwise interpretation such as the Lévy white noises. Thereafter, we mostly rely on the expositions in~\cite{GelVil4} and~\cite{Ito1984foundations}. \\
    
    A \textit{generalized random process} $S$ is a collection of random variable $(\langle S , \varphi \rangle)_{\varphi \in \D(\R)}$ such that, 
    \begin{itemize}
        \item \textbf{Linearity:} For any $\varphi, \psi \in \D(\R)$ and $a,b \in \R$, $\langle S , a  \varphi + b \psi \rangle = a \langle S , \varphi \rangle + b \langle S , \psi\rangle$ almost surely.
        \item \textbf{Continuity:} For any sequence $(\varphi_k)_{k\geq 1}$ of test functions $\varphi_k \in \D(\R)$ such that $\varphi_k \rightarrow \varphi$ in $\D(\R)$ when $k\rightarrow \infty$, the random variables $\langle S , \varphi_k \rangle$ converges in probability towards $\langle S, \varphi \rangle$. 
    \end{itemize}
    
    A generalized random process is therefore a linear and continuous functional $S : \D(\R) \rightarrow L_0(\Omega)$ where $L_0(\Omega)$ is the space of real random variables endowed with the convergence in probability~\cite{Ito1984foundations}.
    We associate to $\D'(\R)$ its \textit{cylindrical $\sigma$-field}, \emph{i.e.} the $\sigma$-field generated by the sets 
    \begin{equation}
    \{ f \in \D'(\R), \ \langle f, \varphi \rangle < a \}
    \end{equation}
    where $\varphi \in \D(\R)$ and $a \in \R$. 
    The probability law of $S$ is then the probability measure $\mathscr{P}_S$ on $\D'(\R)$ such that $\mathscr{P}_S( \{  f \in \D'(\R) , \ \langle f , \varphi \rangle < a \} ) = \mathscr{P} ( \langle S, \varphi \rangle < a )$.
    
    \begin{definition}
    The \emph{characteristic functional} $\CF_S : \D(\R) \rightarrow \C$ of a generalized random process $S$ is defined for $\varphi \in \D(\R)$ by
    \begin{equation}
        \CF_S(\varphi) = \mathbb{E} \left[ \mathrm{e}^{\mathrm{i} \langle S ,\varphi \rangle} \right] = \int_{\D'(\R)} \mathrm{e}^{\mathrm{i} \langle f , \varphi\rangle}  \mathscr{P}_S(\mathrm{d} f) .
    \end{equation}
    \end{definition}
    
    The characteristic functional of $S$ is the (infinite dimensional) Fourier transform of its probability law $\mathscr{P}_S$. 
    It generalizes the characteristic function of random variables to the case of generalized random processes.
    In particular, two generalized random processes have the same law if and only if they share the same characteristic functional. 

    \begin{definition} \label{def:LK}
    We say that $\Psi: \R \rightarrow \C$ is a \emph{characteristic exponent} if it can be written as
    \begin{equation} \label{eq:LK}
        \Psi(\xi) = \mathrm{i} \mu \xi - \frac{\sigma^2\xi^2}{2} + \int_{\R \backslash\{0\}} (\mathrm{e}^{\mathrm{i}\xi t} - 1 - \mathrm{i}\xi t \One_{|t|\geq 1} ) \nu(\mathrm{d} t),
    \end{equation}
    with $\mu \in \R$, $\sigma^2 \geq 0$, and $\nu$ a L\'evy--measure, \emph{i.e.}, a Radon measure over $\R \backslash\{0\}$ such that 
    \begin{equation}
        \label{eq:levymeasurecondition}
        \int_{\R \backslash\{0\}} \inf ( 1, t^2 ) \nu(\drm t) < \infty.
    \end{equation}
    \end{definition}

    A random variable $X$ is \textit{infinitely divisible} if it can be decomposed as a sum $X = X_1 + \cdots + X_N$ of $N$ i.i.d. random variables for any $N \geq 1$. Infinitely divisible random variables are closely related to continuous-domain random processes. The Lévy--Khintchine theorem~\cite[Theorem 8.1]{Sato1994levy} ensures that a random variable $X$ is infinitely divisible if and only if its characteristic function can be written as $\CF_X(\xi) = \exp ( \Psi(\xi) )$ with $\Psi$ a characteristic exponent defined in Definition~\ref{def:LK}.
    Following~\cite[Chapter III]{GelVil4}, we now define Lévy white noise from their characteristic functional and exponent.

    \begin{definition}
    A \emph{L\'evy white noise} is a generalized random process $W$ in $\D'(\R)$ whose characteristic functional is given by 
    \begin{equation}
        \CF_W(\varphi) = \exp \left(  \int_{\R} \Psi( \varphi (x) ) \drm x \right),
    \end{equation}
    where $\Psi$ is a characteristic exponent, called by extension the \textit{characteristic exponent} of $W$. 
    If moreover $W$ is almost surely in $\S'(\R)$, then $W$ is a \emph{tempered Lévy white noise}. 
    \end{definition}
    
    The fact that $\CF_W$ specifies the characteristic functional of some generalized random process in $\D(\R)$ is proved in~\cite[Theorem 5, Section III.4.3]{GelVil4}. It relies on the Bochner--Minlos theorem, which characterizes the characteristic functionals as the continuous and positive-definite functionals $\CF : \D(\R) \rightarrow \C$ such that $\CF(0) = 1$~\cite{Fernique1967processus,Ito1984foundations}.

    The characteristic exponent of the centered Gaussian white noise $W$ is given by $\Psi(\xi) = -\sigma^2 \xi^2 / 2$, which corresponds to \eqref{eq:LK} with   $\mu = 0$ and $\nu = 0$.
    The characteristic functional of the Gaussian white noise is therefore $\CF_W(\varphi) = \exp( - \sigma^2 \lVert \varphi \rVert_2^2 / 2 )$.
    The family of Lévy white noises include the S$\alpha$S white noises and the compound Poisson white noises. These examples, among others, will be considered below. 

    We say that a generalized random process $S$ is \textit{stationary} if $S$ has the same law than its shifted version $S(\cdot + x_0)$ for any $x_0 \in \R$.
    We say that $S$ is \textit{stationary at every point} if the random variables $\langle W,  \varphi \rangle$ and $\langle W, \psi \rangle$ are independent as soon as $\varphi, \psi \in \D(\R)$ have disjoint support (or equivalently, $\varphi (x) . \psi(x) = 0$ for any $x\in \R$).
    
    \begin{proposition} \label{prop:statioindep}
    A Lévy white noise $W$ is stationary and independent at every point. In particular, the random variables $\langle W , \varphi \rangle$ and $\langle W , \varphi(\cdot - x_0) \rangle$ are identically distributed for any $x_0 \in \R$ and are independent for $x_0 \in \R$ such that $\varphi (x) \varphi(x -x_0) = 0$ for any $x \in \R$.
   %Let $W$ be a L\'evy white noise. Then, $W$ is stationary, \emph{i.e.}, $W$ has the same law than its shifted version $w(\cdot + x_0)$ for any $x_0 \in \R$. In particular, for any $\varphi \in \D(\R)$ and $x_0 \in \R$, the random variables $\langle W , \varphi \rangle$ and $\langle W , \varphi(\cdot - x_0) \rangle$ are identically distributed. 
    \end{proposition}
    
    The stationarity follows from the fact that the characteristic functionals $\CF_W$ and $\CF_{W(\cdot- x_0)}$ are equal for any $x_0$. The independence at every point is  proved in~\cite[Section III.4]{GelVil4}. 
    Let $X = \langle W , \varphi \rangle$ with $W$ a Lévy white noise and $\varphi \in \D(\R)$.
    We denote by $\Psi$ and $\nu$ the characteristic exponent and Lévy measure of $W$, respectively. 
    Then, the characteristic function of $X$ is given for any $\xi \in \R$ by 
    \begin{equation} \label{eq:CFX}
        \CF_X(\xi) = \mathbb{E} \left[ \mathrm{e}^{\mathrm{i} \xi X} \right]
        = \exp \left( \int_{\R}  \Psi(\xi \varphi(x)) \drm x \right).
    \end{equation}
    This simply follows from the fact that $\CF_X(\xi) = \mathbb{E} \left[ \mathrm{e}^{\mathrm{i} \xi  \langle W,\varphi \rangle} \right] = \mathbb{E} \left[ \mathrm{e}^{\mathrm{i}  \langle W,\xi \varphi \rangle} \right] = \CF_W(\xi\varphi)$. 
    Then, $X$ is infinitely divisible and its Lévy measure $\nu_X$ can be deduced from the one of $W$ according to the relation
     \begin{equation} \label{eq:thisoneisuseful}
        \int_{\R\backslash\{0\}}  f(t) \nu_X(\drm t) = \int_{\Supp (\varphi)} \int_{\R\backslash\{0\}} f(t\varphi(x)) \nu(\drm t) \drm x \in [0,\infty]
    \end{equation}
    for any measurable and positive function $f : \R\backslash\{0\} \rightarrow \R^+$.
    The relation~\eqref{eq:thisoneisuseful} is proved in~\cite[Theorem 9.1]{Unser2014sparse} and can be recovered from the Lévy--Khintchine expansion \eqref{eq:LK} injected in \eqref{eq:CFX}.

    %------------------------------------------
    \subsection{Tempered Lévy White Noises} \label{sec:temperedlevy}
    %------------------------------------------
    
    The goal of this section is to characterize tempered L\'evy white noises among the complete family in $\D'(\R)$. 
    %We shall recover known results with simplified proofs, such as the characterization with the moments that have been demonstrated in~\cite{Dalang2015Levy,Fageot2014}, and provide a new characterization in terms of discrete white noises.
    
    \begin{theorem}
    \label{theo:lévynoisetempered}
    Let $W$ be a L\'evy white noise in $\D'(\R)$. Then, the following conditions are equivalent. 
    \begin{enumerate}
    
        \item The L\'evy white noise $W$ is tempered; \emph{i.e.}, $w \in \S'(\R)$ almost surely.
        
        \item The characteristic functional $\CF_W$ of $W$ is continuous over 
        $\D(\R)$ for the topology induced by $\S(\R)$.
    
        \item Let $\Psi$ be the characteristic exponent of $W$. Then, there exist  $C > 0$ and $\epsilon > 0$    such that, for any $\xi \in \R$, 
        \begin{equation}
            \abs{\Psi(\xi)} \leq C \left( \lvert \xi \rvert^\epsilon + \lvert \xi \rvert^2 \right).
        \end{equation}
        
        \item For any $\varphi \in \D(\R)$, we have that $X_{W,\varphi} = \left( \langle W, \varphi(\cdot - n) \rangle \right)_{n\in \Z} \in \S'(\Z)$ almost surely.
        
        \item The L\'evy measure $\nu$ of $W$ satisfies that $\int_{|t|\geq 1} |t|^\epsilon \nu(dt) < \infty$ for some $\epsilon > 0$.
        
        \item There exists $\epsilon > 0$ such that, for any $\varphi \in \D(\R)$, $\mathbb{E} [|\langle W , \varphi \rangle|^\epsilon] < \infty$.
        
        \item There exists $\epsilon > 0$ and $\varphi \in \D(\R) \backslash\{0\}$ such that $\mathbb{E} [|\langle W , \varphi \rangle|^\epsilon] < \infty$.
        
    \end{enumerate}
    Moreover, if these conditions are not satisfied, then $W \notin \S'(\Z)$ almost surely. 
    \end{theorem}
    
    \textit{Remarks.} 
    Theorem \ref{theo:lévynoisetempered} provides characterizations of the temperedness of a L\'evy white noise $W$ on its characteristic functional (2.), its characteristic exponent (3.), and its L\'evy measure (5.).
    It also recovers the characterization 6. in terms of positive absolute moments that have been demonstrated in~\cite{Dalang2015Levy,Fageot2014}.
    More precisely, the implications $6. \Rightarrow 3. \Rightarrow 1.$ were shown in~\cite{Fageot2014}, revealing that the existence of positive absolute moment is sufficient for $W$ to be tempered. The later contribution~\cite{Dalang2015Levy} then proved that (i) this condition is also necessary and (ii) that $W$ is almost surely not tempered when 6. do not hold. 
    
    The characterization 7. is a slight variation of 6.: it shows that the moment property is  independent from the test function $\varphi \in \D(\R)\backslash\{0\}$. 
    The key novelty is the characterization 4., which allows us to use the characterization of discrete white noises obtained in Theorem~\ref{theo:discrete}. In particular, this provides a new proof for a Lévy white noise to be tempered or not depending on the existence of positive absolute moments, following an approach different from \cite{Dalang2015Levy}. 

    \begin{proof}[Proof of Theorem~\ref{theo:lévynoisetempered}]
    We will prove the theorem thanks to the relations: 
    \begin{equation}
        1. \Rightarrow  4. \Rightarrow  6. \Rightarrow 7. \Rightarrow  5. \Rightarrow  3. \Rightarrow  2. \Rightarrow 1.
    \end{equation}
    The equivalence between 1. and 2. is  true for any generalized random process and is not restricted to L\'evy white noises. It is established in~\cite[Section III.6.1, Corollary 2]{Fernique1967processus} and will not be covered in this paper. 
    On this matter, see also \cite[Section 2.5]{Ito1984foundations}. For alternative proofs and discussions restricted to Lévy white noises, see~\cite[Theorem 4.1]{Lee2006levy} and \cite[Theorem 4]{Fageot2014}. \\
    
    \textbf{1. $\Rightarrow$  4.} 
    Let $\varphi \in \D(\R)$. We fix $(v_k)_{k\geq 1}$ such that $v_k = (v_k[n])_{n\in \Z} \in \D(\Z)$ for any $k \geq 1$ and $v_k \rightarrow 0$ in $\S(\Z)$ when $k \rightarrow \infty$. This means that $\lVert v_k \rVert_{\infty,\alpha} \rightarrow 0$ when $k \rightarrow \infty$ for any $\alpha \geq 0$. 
    We observe that 
    \begin{equation} \label{eq:moonshadow}
        \langle X_{W,\varphi} , v_k \rangle 
        = \sum_{n\in \Z} \langle W, \varphi (\cdot - n) \rangle v_k[n] 
        = \left\langle W , \sum_{n\in\Z} v_k[n] \varphi(\cdot - n )\right\rangle.
    \end{equation}  
    Note that the sums in \eqref{eq:moonshadow} are finite and $\varphi_k = \sum_{n\in\Z} v_k[n] \varphi(\cdot - n ) \in \D(\R)$ for $v_k \in \D(\Z)$ and $\varphi \in \D(\R)$. We will prove that 
    \begin{equation}
        \label{eq:whattoproveonphik}
        \varphi_k \underset{k\rightarrow \infty}{\longrightarrow} 0 \text{ in }\S(\R).
    \end{equation}
    Assume first that \eqref{eq:whattoproveonphik} is established. Since $W \in \S'(\Z^d)$ a.s., we therefore have that, a.s., $\langle X_{W,\varphi}, v_k \rangle = \langle W, \varphi_k \rangle \rightarrow 0$ as $k \rightarrow \infty$. Thanks to Proposition~\ref{prop:caracSpinDp}, this implies that $X_{W,\varphi} \in \S'(\Z)$ almost surely and 4. is proved. \\
     
    Let $\alpha \geq 0$ and $\beta \in \N$. 
    %We set $\alpha_k = \lVert v_k \rVert_{\infty,\alpha}$, which vanishes when $k\rightarrow \infty$. 
    Then, for any $x\in \R$,
    we have   
    \begin{align} \label{eq:bigcomputation}
        |x|^\alpha |\mathrm{D}^\beta \varphi_k (x)|
        &= 
        |x|^\alpha \left|\sum_{n\in \Z} v_k[n] (\mathrm{D}^\beta \varphi_k) (x - n)\right|  
        \leq 
        \sum_{n\in \Z} \left| v_k[n]  |x|^\alpha (\mathrm{D}^\beta \varphi_k) (x - n)\right| \nonumber \\
        &\leq  \lVert v_k \rVert_{\infty,\alpha+2}
        \left( \sum_{n\in \Z} \frac{ (1 + |x|)^\alpha}{ (1 + |n|)^{\alpha+2}}  \left| (\mathrm{D}^\beta \varphi_k) (x - n)\right| \right)  \nonumber \\
        &\leq  
        \lVert v_k \rVert_{\infty,\alpha+2}
        \lVert (1+|\cdot|)^\alpha \mathrm{D}^\beta \varphi \rVert_\infty 
        \left( \sum_{n\in \Z} \frac{ (1 + |x|)^\alpha}{ (1 + |n|)^{\alpha+2} (1+ |x-n|)^\alpha } \right). 
    \end{align}
%    The last bound in \eqref{eq:bigcomputation} only depends on $k$ via $ \lVert v_k \rVert_{\infty,\alpha+2}$, which vanishes when $k\rightarrow 0$. It therefore suffices to show that the sum in the last inequality is finite to deduce the result.
    We define, for $x,y \in \R$, 
    \begin{equation}
        F(x,y) = \frac{(1 + |y-x| )^\alpha}{(1 + |x| )^\alpha(1 + |y| )^\alpha}.
    \end{equation}
    If $x$ and $y$ have the same sign, then $|x-y| \leq  \max(|x|,|y|)$ and therefore $1 + |x-y| \leq 1 + \max(|x|,|y|) \leq (1+|x|)(1+|y|)$. If now $x$ and $y$ have opposite signs, for instance $y \leq 0 \leq x$, then 
    $1 + |x-y| = 1 + x-y \leq 1+ x - y -xy = (1+x) (1-y) = (1+|x|)(1+|y|)$. This shows that, $0 < F(x,y) \leq 1$ for anx $x,y \in \R$. 
    Finally, this implies that 
    \begin{equation}
    \label{eq:lastdudeforthisstuff}
        \sum_{n\in \Z} \frac{ (1 + |x|)^\alpha}{ (1 + |n|)^{\alpha+2} (1+ |x-n|)^\alpha } 
        =
        \sum_{n\in \Z} \frac{F(n, x-n)}{ (1 + |n|)^{2}} 
        \leq 
        \sum_{n\in \Z} \frac{1}{ (1 + |n|)^{2}} = \frac{\pi^2}{3} < \infty.
    \end{equation}
    Combining \eqref{eq:bigcomputation} and \eqref{eq:lastdudeforthisstuff}, we just proved that
    \begin{equation}
        \lVert \varphi_k \rVert_{\infty,\alpha,\beta}  =  \sup_{x\in \R}  |x|^\alpha |\mathrm{D}^\beta \varphi_k (x)| \leq C \lVert v_k \rVert_{\infty,\alpha+2}
    \end{equation}
    for some constant $C > 0$ independent from $k \geq 1$. 
    Since $\lVert v_k \rVert_{\infty,\alpha + 2}$ vanishes for any $\alpha \geq 0$, we deduce that $\lVert \varphi_k \rVert_{\infty,\alpha,\beta} \rightarrow 0$ as $k \rightarrow \infty$ for any $\alpha \geq 0$ and $\beta \in \N$. Therefore $\varphi_k \rightarrow 0$ in $\S(\R)$ when $k\rightarrow \infty$ and $X_{W,\varphi} \in \S'(\Z)$ almost surely. \\
    
    \textbf{4. $\Rightarrow$  6.}
    We first assume that $\varphi \in \D(\R)$ has its support included in $[0,1]$, so that the test functions $\varphi(\cdot - n)$ have disjoint supports for $n\in\Z$. Hence, the random variables $\langle W, \varphi(\cdot - n)\rangle$ are independent and identically distributed for $n\in \Z$ (see Proposition~\ref{prop:statioindep}). We  apply Theorem~\ref{theo:discrete} and deduce that $\mathbb{E} [ |\langle W , \varphi \rangle|^\epsilon ] < \infty$ for some $\epsilon > 0$.
    
    For a general $\varphi \in \D(\R)$ with no support restriction, the random variables $\langle W, \varphi(\cdot - n)\rangle$ are not independent in general. 
    However, the function $\varphi$ being compactly supported, there exists $K_0 \geq 1$ such that the functions $\varphi(\cdot  + K_0 n)$, $n\in \Z$, have disjoint supports. We can apply the same argument to the sequence $(\langle W,  \varphi(\cdot - K_0 n) \rangle)_{n\in \Z}$ and 6. is proved. \\

    \textbf{6. $\Rightarrow$ 7. $\Rightarrow$ 5.} 
    The implication 6. $\Rightarrow$ 7. is obvious. We assume that $\varphi \neq 0$ and $\epsilon > 0$ are such that $\mathbb{E} [|\langle W , \varphi \rangle|^\epsilon] < \infty$. The key is the connection between the moments of an infinitely divisible random variable $X$ and the moments of its L\'evy measure $\nu_X$. According to~\cite[Theorem 25.3]{Sato1994levy}, we indeed have that, for any $p> 0$, 
    \begin{equation} \label{eq:satomoment}
        \mathbb{E} [|X|^p] < \infty \Longleftrightarrow \int_{|t|\geq 1} |t|^p \nu_X(\drm t) < \infty. 
    \end{equation}
    We apply this relation to $X = \langle W , \varphi \rangle$ and $p=\epsilon$ to deduce that $\int_{|t|\geq 1} |t|^\epsilon \nu_{\langle W , \varphi \rangle} (\drm t) < \infty$. 
    Applying \eqref{eq:thisoneisuseful} with $f(t) = |t|^\epsilon \One_{|t|\geq 1}$ then reveales that
    \begin{equation}
       \int_{|t| \geq 1}  |t|^\epsilon \nu_{\langle W , \varphi \rangle}(\drm t) = \int_{\R\backslash\{0\}} \left( \int_{\{x \in \R, \  | t  \varphi(x) | \geq 1\} } |t \varphi(x)|^\epsilon \drm x \right) \nu(\drm t).  
    \end{equation}
    The function $\varphi$ being non identically zero, 
    there exists an interval $[a,b]$ and $m > 0$ such that $|\varphi(x)| \geq m$ for $a \leq x \leq b$. We then deduce that
    \begin{align} \label{eq:trucutile233}
       \int_{|t| \geq 1}  |t|^\epsilon \nu_{\langle W , \varphi \rangle}(\drm t) 
       &=\int_{\R\backslash\{0\}} \left( \int_{a}^b  \One_{|\varphi(x)| |t| \geq 1} |t|^\epsilon \drm x \right) \nu(\drm t) \geq 
       m^\epsilon \int_{\R\backslash\{0\}} \int_{a}^b  \One_{m |t| \geq 1} |t|^\epsilon \nu(\drm t) \drm x \nonumber \\
       &= m^\epsilon (b-a) \int_{|t| \geq 1/m} |t|^\epsilon \nu(\drm t)
    \end{align}
    where we used that $\One_{|t \varphi(x) | \geq 1} \geq \One_{m |t| \geq 1}$ for any $x \in [a,b]$. 
    The finiteness of 
    $\int_{|t| \geq 1} |t|^\epsilon \nu_{\langle W , \varphi \rangle}(\drm t)$   therefore implies that
    $\int_{|t| \geq 1/m} |t|^\epsilon \nu(\drm t)< \infty$. Finally, since $\int_{|t| \geq 1} |t|^\epsilon \nu(\drm t)$ and $\int_{|t| \geq 1/m} |t|^\epsilon \nu(\drm t)$ differ by a finite constant (given by $c = \int_{\min(1,1/m) \leq |t| \leq \max(1,1/m)} |t|^p \nu(\drm t)$), we deduce 5. \\

    \textbf{5. $\Rightarrow$  3.} 
    This is a reformulation of~\cite[Proposition 2.4]{fageot2017gaussian}. We repeat the proof for the sake of self-completeness. Note that if $\epsilon > 1$, then $\int_{|t|\geq 1} |t| \nu(\drm t ) \leq \int_{|t|\geq 1} |t|^\epsilon \nu(\drm t ) < \infty$. Hence, we can assume that $\epsilon \leq 1$ without loss of generality. 
    We use the Lévy--Khintchine expansion \eqref{eq:LK} of $\Psi$. The function $\xi \mapsto \mathrm{i} \mu \xi - \frac{\sigma^2\xi^2}{2}$ is clearly dominated by $\xi \mapsto |\xi|^\epsilon + |\xi|^2$ (since $\epsilon \leq 1$).
    It then suffices to bound the last term in \eqref{eq:LK}. 
    
    Applying the inequality $|\mathrm{e}^{\mathrm{i} \xi t} - 1 - \mathrm{i} \xi t|\leq \xi^2 t^2$, we deduce that
    \begin{equation}
        \int_{0< |t|< 1} \left| \mathrm{e}^{\mathrm{i} \xi t} - 1 - \mathrm{i} \xi t \right| \nu(\drm t) \leq \xi^2 \int_{0< |t|< 1} t^2 \nu(\drm t). 
    \end{equation}
    The quantity $\int_{0< |t|< 1} t^2 \nu(\drm t)$ is finite for any Lévy measure, hence this term is also dominated as expected. 
    
    We now use the inequality $|\mathrm{e}^{\mathrm{i} \xi t} - 1|^2  = 2( 1 - \cos(\xi t)) \leq 2 \min ( 2, \xi^2 t^2) \leq 4 \min ( 1 , \xi^2 t^2) \leq 4 |\xi t|^{2\epsilon}$ (using that $2 \epsilon \leq 2$). This implies that
    \begin{equation}
        \int_{|t|\geq 1} \left| \mathrm{e}^{\mathrm{i} \xi t} - 1  \right| \nu(\drm t)
        \leq 
        2 |\xi|^\epsilon \int_{|t|\geq 1}  |t|^\epsilon \nu(\drm t). 
    \end{equation}
    Combining the different bounds, we obtain 3.\\

    \textbf{3. $\Rightarrow$  2.} 
    %This can be deduced from the main results of~\cite{Fageot2014}, in particular Theorem 3. Again, we provide a self-contained and simplified proof.
    The characteristic functional is positive-definite and normalized such that $\CF_W(0) = 1$. Hence, we deduce that~\cite[Section 13.4.B]{loeve1977elementary} 
    \begin{equation}
     \left\rvert \CF_W(\varphi) - \CF_W(\psi) \right\rvert \leq
     2 \left( 1 - \Re \left\{ \CF_W(\varphi-\psi) \right\} \right)
    \end{equation}
    for any $\varphi, \psi \in \D(\R)$, where $\Re \{z\}$ stands for the real part of $z\in \C$. In particular, $\CF_W$ is continuous if and only if it is continuous at $\varphi = 0$, which we will prove now. 
    The condition 3. implies that 
    \begin{equation}
        \left\lvert   \int_{\R}  \Psi(\varphi(x)) \drm x   \right\rvert \leq 
        \int_{\R} |\Psi(\varphi(x))| \drm x 
        \leq 
        C ( \lVert \varphi \rVert_\epsilon^\epsilon + \lVert \varphi \rVert_2^2).
    \end{equation}
    Assume that $\varphi_k \rightarrow 0$ as $k\rightarrow \infty$ for the topology of $\S(\R)$. 
    This readily implies that $\lVert \varphi_k \rVert_\epsilon^\epsilon + \lVert \varphi_k \rVert_2^2 \rightarrow 0$, 
    and therefore $\left\lvert   \int_{\R}  \Psi(\varphi_k(x)) \drm x   \right\rvert \rightarrow 0$. 
    Finally, 
    \begin{equation}
    \left\rvert \CF_W(\varphi_k) - \CF_W(0) \right\rvert = 
         \left\rvert \CF_W(\varphi_k) - 1 \right\rvert  = 
          \left\rvert \mathrm{e}^{\int_{\R} \Psi(\varphi_k(x)) \drm x} - 1
          \right\rvert \underset{k\rightarrow \infty}{\longrightarrow} |\mathrm{e}^0 -1 | = 0,
    \end{equation}
    proving 2. \\
%    by using that $x \mapsto 1-\mathrm{e}^{-x}$ is increasing and such that $0 \leq 1-\mathrm{e}^{-x} \leq x$ for $x \geq 0$. 
%    Then, the condition 3. implies that $\int_{\R} |\Psi(\varphi(x))| \drm x \leq C ( \lVert \varphi \rVert_\epsilon^\epsilon + \lVert \varphi \rVert_2^2)$. We finally conclude by observing that the convergence $\varphi \rightarrow 0$ for the topology of $\S(\R)$ implies that $\lVert \varphi \rVert_\epsilon^\epsilon + \lVert \varphi \rVert_2^2 \rightarrow 0$, and therefore $\CF_W(\varphi) \rightarrow 0$. 

We assume finally that the equivalent conditions 1. to 7. are not satisfied. Set $\Omega_{W}= \{ \omega \in \Omega , \ W(\omega) \in \S'(\R)\}$. Then, the proof above for the implication $1. \Rightarrow 4.$ implies that the sequence $X_{W(\omega),\varphi} \in \S'(\Z)$ for $\omega \in \Omega_W$. In other terms, we have that $\mathscr{P}(\Omega_W) \leq \mathscr{P} (X_{W,\varphi} \in \S'(\Z)\}$. 
However, Theorem~\ref{theo:discrete} implies that $X_{W,\varphi} \notin \S'(\Z)$ a.s. as soon as it is not a.s. tempered, which is the case here. Hence,
\begin{equation}
    \mathscr{P}(\Omega_W) \leq \mathscr{P} (X_{W,\varphi} \in \S'(\Z)\} = 0
\end{equation}
and $\mathscr{P}(\Omega_W) = 0$; \emph{i.e.}, $W \notin \S'(\R)$ a.s. 
    \end{proof}
    
   As we have seen, the first characterization of tempered L\'evy white noises in terms of moment properties has been obtained in~\cite{Dalang2015Levy}. The proof's strategy of Dalang and Humeau is based on (i) the Lévy--Itô decomposition of a Lévy process, (ii) the reduction of the problem to compound Poisson processes, and (iii) the deduction of the corresponding result for Lévy white noises, which are the weak derivatives of Lévy processes).
   
    Our proof rather relies on the random sequences $X_{W,\varphi} = \left( \langle W, \varphi(\cdot - n)\rangle \right)_{n\in \Z}$ and on the fact that tempered i.i.d. random sequences have been characterized in Theorem~\ref{theo:discrete}. This new proof highlights the strong connection between discrete and Lévy white noises. \\
    
    We illustrate Theorem~\ref{theo:lévynoisetempered} on classical examples of Lévy white noises.
    The terminology for the noise $W$ is given by the corresponding infinitely divisible random variable $X$ with identical characteristic exponent. \\
    
    \emph{S$\alpha$S white noises.} The S$\alpha$S case has been studied in~\cite{Fageot2017besov,Huang2007fractional} in the framework of tempered generalized random processes.
    A  S$\alpha$S white noise $W_\alpha$ corresponds to the characteristic exponent $\Psi(\xi) = - \gamma^\alpha |\xi|^\alpha$ with $0< \alpha \leq 2$ and $\gamma > 0$. The characteristic functional of $W_\alpha$ is then 
    \begin{equation}
        \CF_{W_\alpha}(\varphi) = \exp( - \gamma^\alpha \lVert \varphi \rVert_\alpha^\alpha) 
    \end{equation}
        for any $\varphi \in \D(\R)$. The criterion 3. of Theorem~\ref{theo:lévynoisetempered} directly recovers that $W_\alpha \in \S'(\R)$ almost surely. The Gaussian white noise corresponds to $\alpha = 2$.  \\
        
     \emph{Compound Poisson white noises.} 
     A compound Poisson white noise $W$ is a point process such that
     \begin{equation}
         W = \sum_{n\in \Z} A_n \delta( \cdot - Y_n)
     \end{equation}
     where $A=(A_n)_{n\in \Z}$ is a sequence of i.i.d. random variables with common law $\mathscr{P}_A$ and where the $Y_n$ are such that $N = \mathrm{Card}(n \in \Z, \ a \leq Y_n \leq b )$ is a Poisson random variable of parameter $\lambda (b-a)$ for some $\lambda >0 $ and any $a<b$. We call $\mathscr{P}_A$ the law of the jumps of $W$, which can be any probability law with support in $\R\backslash\{0\}$. 
     The compound Poisson white noises can be used as the approximation in law of any Lévy white noises~\cite{fageot2018gaussian}. This principle has been used to generate Lévy white noises and their extensions such as Lévy processes~\cite{dadi2020generating}. 
    The characteristic exponent of $W$ is given by~\cite[Theorem 1]{Unser2011stochastic}
        \begin{equation}
            \Psi(\xi) = \exp\left( \lambda ( \CF_{A}(\xi) - 1 ) \right)
        \end{equation}
    where $\CF_{A}$ is the characteristic function of the $A_n$. 
    
    Let $X$ be a compound Poisson random variable with characteristic exponent $\Psi$. Then, we have the equivalence 
    \begin{equation}\label{eq:leptittruc}
        \mathbb{E} [|X|^p] < \infty \quad \Longleftrightarrow \quad \mathbb{E} [|Y_0|^p] < \infty
    \end{equation}
    for any $p > 0$. 
    Indeed, the Lévy measure of $X$ is $\lambda \mathscr{P}_A$ and therefore \eqref{eq:leptittruc} readily follows from~\eqref{eq:satomoment}. This implies the following result for compound Poisson white noise.
    
    \begin{proposition} \label{prop:CPwn}
    Let $W$ be a compound Poisson white noise in $\D'(\R)$ with parameter $\lambda > 0$ and law of jump $\mathscr{P}_A$. Then, $W$ is tempered if and only if there exists $\epsilon > 0$ such that
    \begin{equation}
        \int_{\R\backslash\{0\}} |x|^\epsilon \mathscr{P}_A(\drm x) < \infty.    
    \end{equation}
    Otherwise, $W \notin \S'(\R)$ almost surely.
    \end{proposition}

%    In particular, we deduce that a Lévy white noise is tempered if and only if its law of jump $\mathscr{P}_A$ has a finite absolute moment for some $\epsilon > 0$.  \ju{Question: make a proposition out of it?}
    
    \emph{Some tempered Lévy white noises.} Any infinitely divisible law specifies a Lévy white noise with identical characteristic exponent. 
    Student's t-distributions and Laplace laws are infinitely divisible. 
    We can therefore consider Student's t-white noises~\cite{zhang2015tar}, Laplace white noises~\cite{Koltz2001laplace}, which are all known to admits to positive absolute moments and are therefore tempered. This is also the case for layered stable white noises~\cite{Houdre2007layered}, tempered stable white noises~\cite{grabchak2016tempered,Rosinski2007tempering}, and inverse Gaussian white noises~\cite{Barndorff1997processes}. \\
    
    \emph{Non-tempered Lévy white noises.}
    In order to construct Lévy white noises which are almost surely not tempered, it suffices to consider Lévy measures which do not satisfy the condition 5. of Theorem~\ref{theo:lévynoisetempered}. An example is given by the measure  \begin{equation}
    \nu(\drm t) = \frac{\drm t}{ (1 + |t|) \log^2(2 + |t|)},    
    \end{equation}
    which satisfies $\int_{\R\backslash\{0\}} \min( 1 , t^2) \nu(\drm t) < \infty$ and $\int_{\R\backslash\{0\}} |t|^\epsilon \nu(\drm t) = \infty$ for any $\epsilon > 0$. Alternatively, any compound Poisson white noise whose law of jump $\mathscr{P}_A$ has no positive absolute moment is  almost surely not tempered according to Proposition~\ref{prop:CPwn}. \\
    
    \emph{Connection with the Pruitt index.}
    In~\cite{Pruitt1981growth}, W.E. Pruitt studied the growth properties of random walks and Lévy processes. His result can be reinterpreted by considering the \emph{Pruitt index} of the Lévy white noise $W$ with characteristic exponent $\Psi$, defined as\footnote{The usual definition of the Pruitt index is different, but its equivalence with \eqref{eq:beta0def} is for instance proved in~\cite[Setion 5]{Schilling1998growth}.}
    \begin{equation} \label{eq:beta0def}
        \beta_0 = \sup \left\{ p \geq 0, \ \underset{ |\xi| \rightarrow 0 }{\lim \sup} \frac{|\Psi(\xi)|}{|\xi|^p} < \infty \right\}.  
    \end{equation}
    According to \cite[Proposition 48.10]{Sato1994levy}, denoting by $\mu$ the Lévy measure of $W$, we have that
$\beta_0 = \sup \left\{ 0\leq  p \leq 2, \ \int_{\R\backslash \{0\} } |t|^p \nu(\mathrm{d}t)  < \infty\right\}.$
We therefore see with Theorem~\ref{theo:lévynoisetempered} that
    \begin{equation}
        W \in \S'(\R) \text{ a.s. } \Longleftrightarrow \  \beta_0 > 0 \quad \text{and} \quad W \notin \S'(\R) \text{ a.s. } \Longleftrightarrow \ \beta_0 = 0.
    \end{equation}

    \section{Conclusion}
    
    We provided a new and self-contained proof characterizing the temperedness of Lévy white noises. We recovered the results of previous works, including~\cite{Dalang2015Levy,fageot2017gaussian,Fageot2014}. 
    Our analysis highlights the connection between the Lévy white noises and their discrete counterparts, which are infinite sequences of independent and identically distributed random variables. In particular, we characterized the location of discrete white noises in polynomially weighted $\ell_\infty$-spaces in terms of their absolute moments in Theorem \ref{theo:discrete}. 
    %We believe this result to be interesting on its own.
    We finally used the discrete characterization in our analysis of the tempered Lévy white noises in Theorem~\ref{theo:lévynoisetempered}.
    In future works, we plan to extend the proposed considerations to new spaces of generalized functions, including tempered ultra-distributions~\cite{triebel1983ultra} and Fourier hyperfunctions~\cite{kaneko1989introduction}. 

\bibliographystyle{plain}
\bibliography{references}

\end{document}